\documentclass{amsart}

\usepackage{url,amssymb,enumerate,colonequals}

\usepackage{mathrsfs} 
\usepackage[section]{placeins}
\usepackage{MnSymbol}
\usepackage{extarrows}
\usepackage{lscape}
\usepackage[all,cmtip]{xy}

\usepackage[OT2,T1]{fontenc}
\usepackage{color}

\usepackage[
        colorlinks, citecolor=darkgreen,
        backref,
        pdfauthor={Elvira Lupoian},
]{hyperref}

\usepackage{comment}
\usepackage{multirow}
\usepackage{hyperref}

\newcommand{\Q}{\mathbb{Q}}

\newcommand{\Z}{\mathbb{Z}}

\DeclareMathOperator{\Gal}{Gal}

\newcommand{\union}{\cup} 

\numberwithin{equation}{section}

\newtheorem*{thm}{Theorem}
\newtheorem{lemma}{Lemma}
\newtheorem{conj}{Conjecture}

\newtheorem{proposition}{Proposition}

\theoremstyle{definition}

\theoremstyle{remark}
\newtheorem{remark}[equation]{Remark}

\newenvironment{psmallmatrix}
  {\left(\begin{smallmatrix}}
  {\end{smallmatrix}\right)}

\definecolor{darkgreen}{rgb}{0,0.5,0}

\DeclareRobustCommand{\SkipTocEntry}[5]{}

\begin{document}

\title{Computing the Cuspidal Subgroup of the Modular Jacobian $J_{H}\left( p \right)$ }

\author{Elvira Lupoian}

\address{Mathematics Institute\\
    University of Warwick\\
    CV4 7AL \\
    United Kingdom}

\email{e.lupoian@warwick.ac.uk}
\date{\today}
\thanks{The author is supported by the EPSRC studentship}
 
\begin{abstract}
 For a fixed prime $p$ congruent to $1$ modulo $4$ we may define the modular curve $X_{H}\left( p \right)$ associated to the subgroup of non-zero squares modulo $p$. This curve has four cusps and we consider the subgroup of the Jacobian $J_{H}\left( p \right)$ of $X_{H}\left( p \right)$ generated by these points, which we will call the cuspidal subgroup of $J_{H}\left( p \right)$. This is a finite subgroup by the results of Manin and Drinfeld, and lies inside the $\Q \left( \sqrt{ p} \right)$-rational torsion subgroup. In this paper we compute the cuspidal subgroup for all such curves of genus $g$, $2 \leq g \leq 10$, namely those with $p \in \{ 29, 37, 41, 53, 61, 73 \}$, and compare this with  $J_{H}\left( \Q \left( \sqrt{p} \right) \right)_{\text{tors}}$.
\end{abstract}
\maketitle

\section{Introduction}
For any positive integer $N$, let $J_{0}\left( N \right)$ be the Jacobian variety of $X_{0}\left(N \right)$, the canonical model, defined over $ \mathbb{Q}$, of the modular curve of level $\Gamma_{0}\left( N \right) $.
The famous Mordell-Weil theorem tells us that $J_{0}\left( N \right) \left( \mathbb{Q} \right)$ is a finitely generated group; that is,
\begin{center}
  $ J_{0}\left( N \right) \left( \mathbb{Q} \right) \cong J_{0}\left( N \right) \left( \mathbb{Q} \right)_{\textbf{tor}} \oplus \mathbb{Z}^{r}  $  \end{center}
 for some integer $r \geq 0$, and a finite group $ J_{0}\left( N \right) \left( \mathbb{Q} \right)_{\textbf{tor}}$, known as the rational torsion subgroup of $J_{0}\left( N \right)$. The rank of $J_{0}\left( N \right) $ over $\Q$ can be analysed using the decomposition of $J_{0}\left( N \right)$ as a product of abelian varieties of $GL_{2}$-type; for some details see \cite{Stein}. Due to the additional structure of the modular curve, many things can also be said about the structure of the finite torsion subgroup $ J_{0}\left( N \right) \left( \mathbb{Q} \right)_{\textbf{tor}}$.

 For prime level $p \geq 5$, this group is completely understood due to the work of Mazur \cite{Mazur}.

\begin{thm}{(Mazur)}
For a prime $p \geq 5$, $ J_{0}\left( N \right) \left( \mathbb{Q} \right)_{\textbf{tor}}$ is a cyclic subgroup of order the numerator of $\frac{p-1}{12}$, and it is generated by the linear equivalence class of the difference of $0$ and $\infty$, the two cuspidal points of $X_{0}\left( p \right)$.
\end{thm}

There is a natural generalisation of this question for non-prime level. For a positive integer $N$, let $C_{N}$ be the subgroup of $J_{0}\left( N \right) $ generated by linear equivalences of differences of cusps of $X_{0}\left(N \right)$, which we call the cuspidal subgroup of $J_{0}\left( N \right)$. We denote by $C_{N} \left( \mathbb{Q} \right)$ the rational cuspidal subgroup of $J_{0}\left( N \right)$; that is elements of $C_{N}$ which are invariant under the action of the absolute Galois group $G_{\bar{\mathbb{Q}}} = \Gal \left( \bar{\mathbb{Q}} / \mathbb{Q} \right)$. 

Theorems of Manin \cite{Manin} and Drinfeld \cite{Drinfeld} show that the difference of two cusps is torsion and hence
\begin{center}
$C_{N} \left( \mathbb{Q} \right) \subseteq  J_{0}\left( N \right) \left( \mathbb{Q} \right)_{\textbf{tors}} $.
\end{center}
As a generalization of Mazur's results, known as the Ogg Conjecture, prior to \cite{Mazur}, the following statement is conjectured. 

\begin{conj}{(Generalized Ogg Conjecture)}
For any positive integer $N$, we have
\begin{center}
$C_{N} \left( \mathbb{Q} \right) =   J_{0}\left( N \right) \left( \mathbb{Q} \right)_{\textbf{tors}} $ 
\end{center}
\end{conj}

There is plenty of evidence supporting this conjecture; some of which is computational, including for $N \in \{ 11, 14, ,15, 17, 19, 20, 21, 24, 27, 32, 36, 49\}$ by Ligozat \cite{Ligozat}, for $N = 125 $ by Poulakis \cite{Poulakis} and  for $N \in \{ 34, 38, 44, 45, 51, 52, 54, 56, 64, 82 \} $ by Ozman and Siksek \cite{OS}.

There are also a number of theoretical result supporting this. 
For instance Ribet and Wake \cite{RibetWake} proved that for a square-free $N$ and any prime $p \nmid 6N$, the p-primary parts of $J_{0}\left(N \right) \left( \mathbb{Q} \right)_{\text{tors} }$ and $C_{N} \left( \mathbb{Q} \right)$ coincide. Many other similar results have been proved, see for instance \cite{ren}, \cite{yoo2023} and \cite{yoo2015}.

Naturally, we can ask the same question for other curves. For any modular curve, the equivalence class of the difference of cusps  has finite order on the Jacobian, and hence the cuspidal subgroup is a subgroup of the $K$-rational torsion points of the Jacobian, where $K$ is the field of definition of the cusps, and we may ask whether this inclusion is in fact an equality. A first step usually involves understanding the structure of the cuspidal subgroup; a problem that has been studied by many authors. 
For instance, the cuspidal subgroup of the  non-split Cartan modular curves of prime power level was studied by Carlucci in \cite{carlucci}. There are also examples of subgroups of the cuspidal subgroup being studied. For instance, Yang \cite{yang2007modular} gives a formula for the size of the subgroup of the cuspidal subgroup of $X_{1}\left( N \right)$ generated by the cusps lying above the cusp $\infty \in X_{0}\left( N \right)$, and when the level is prime the size of the entire cuspidal subgroup is given by Takagi \cite{takagi1992cuspidal}. The size of the subgroup of the cuspidal group generated by equivalence classes of divisors fixed by the Galois action, for some modular curves of prime level lying between $X_{0}\left( p \right)$ and $X_{1}\left( p \right)$ has been studied by Chen \cite{chen2011cuspidal}. These questions are often difficult to study as they involve computing the group (or some specific subgroup) of the modular units of the modular curve. An overview of the general, classical method used to study cuspidal subgroups is given by Gekeler \cite{gekeler2011cuspidal}.

In this paper we will study the cuspidal subgroup and its relations to the larger torsion subgroup in which it lies, for some intermediate modular curves. Let $p$ be an odd prime and $H$ a subgroup of the multiplicative group $\left( \mathbb{Z} /p\mathbb{Z} \right)^{*} $. Let $\Gamma_{H}\left( p \right)$ be the congruence subgroup associated to $H$,  $X_{H}\left( p \right) $ the associated modular curve and $J_{H}\left( p \right)$ the Jacobian variety of the curve, see Section 2 for definitions. Let $C_{H}\left( p \right)$ be the subgroup of $J_{H}\left( p \right) $ generated by linear equivalence classes of differences of cusps, which we will call the cuspidal subgroup as before. In this paper, we take $H$ to be subgroup of squares modulo $p$, for a large enough prime $p$, which is additionally congruent to $1$ modulo $4$, and we explicitly investigate the above question. Note that for primes $p   \equiv  3 \ \text{mod} \  4$, the two modular curves $X_{H}\left( p \right)$ and $X_{0}\left( p \right)$ coincide, and hence the above inclusion is an equality by the Mazur's theorem.
The main result proved in Section 4 is the following.

\begin{thm}
Let $p \geq 5$ be a prime such that the genus $g$ of $X_{H}\left( p \right)$ satisfies $1 \leq g \leq 10$. Then 
\begin{center}
$C_{H}\left( p \right)  \left( \mathbb{Q} \right) =   J_{H} \left( p \right) \left( \mathbb{Q} \right)_{\textbf{tors}} $
\end{center}
Furthermore, if additionally $p \ \equiv \ 1 \  mod \ 4 $, that is $p \in \{ 29, 37, 41, 53, 61, 73 \}$,
\begin{center}
$C_{H}\left( p \right) =  \cong J_{H} \left( p \right) \left( \mathbb{Q} \left( \sqrt{p} \right) \right)_{\textbf{tors}}\left( \mathbb{Z} /n \mathbb{Z} \right) \times \left( \mathbb{Z} /m \mathbb{Z} \right) $ \\
$C_{H}\left( p \right)  \left( \mathbb{Q} \right) =   J_{H} \left( p \right) \left( \mathbb{Q} \right)_{\textbf{tors}}\cong  \left( \mathbb{Z} /m \mathbb{Z} \right) $ 
\end{center}
where $n$ is positive integers depending on the level p,
\begin{center}
\begin{tabular}{ | c | c | c | }
\hline 
 p & n & m  \\ 
 \hline 
 29 & 3 & 21 \\  
 37 & 5 & 15 \\
 41 & 8 & 40 \\
 53 & 7 & 91 \\
 61 & 11 & 55  \\
 73 & 22 & 66 \\
 \hline 
\end{tabular}
\end{center}
and $m$ is the lowest common multiple of $n$ and the numerator of $ \ \frac{p-1}{12}$.
\end{thm}

All the \texttt{MAGMA} code used to compute to cuspdial subgroups and torsion subgroups of the Jacobians of the curves stated above can be found in the \texttt{GitHub} repository:
\begin{center}
\href{ https://github.com/ElviraLupoian/XHSquares}{https://github.com/ElviraLupoian/XHSquares}
\end{center}

The main motivation for studying the modular curves $X_{H}\left( p \right)$ is their moduli interpretation. As with the classical modular curves, non-cuspidal points of the $X_{H}\left( p \right)$ parameterise elliptic curves with certain $p$-torsion information. In particular, non-cuspidal points of $X_{H}\left( p \right)$ represent elliptic curves $E$ whose image of the  mod $p$ Galois representation is of the form 
\begin{center}
    $\bar{\rho}_{E,p} \sim \begin{psmallmatrix}
        * \in H & * \\ 0 & * 
    \end{psmallmatrix}.$
\end{center}
For a summary of this see \cite{siksek2019explicit}.
We might be able to understand the rational points of $X_{H}\left( p \right)$ by embedding them into the rational points of $J_{H}\left( p \right)$, and hence understanding $J_{H}\left( p \right)\left( \Q \right)_{\text{tors}}$, or part of this, is an important part of understanding $J_{H}\left( p \right) \left( \Q \right)$ and $X_{H}\left( p \right)\left( \Q \right)$. 
\section{Preliminaries}
In this section we give an overview of some basic facts and constructions related to modular curves. For a more comprehensive overview the reader may refer to \cite[Chapter 1]{shbook} or \cite[Chapter 1]{stevensarith}.

Fix a congruence subgroup $\Gamma$ of $SL_{2}\left( \Z \right)$.  This group acts on the upper half planes $\mathbb{H} = \{ z \in \mathbb{C} \ |  \ \text{Im} \left( z \right) > 0 \}$ via fractional linear transformations,
\begin{center}
  $  \begin{psmallmatrix}
        a & b \\ c & d
    \end{psmallmatrix} \cdot \tau = \frac{a\tau + b }{c \tau + d } \  \text{for any } \ \begin{psmallmatrix}
        a & b \\ c & d
    \end{psmallmatrix}  \in \Gamma, \tau\in \mathbb{H}$
\end{center}
and the resulting quotient  $Y\left( \Gamma \right) =  \Gamma / \mathbb{H}$ is a complex Riemann surface, which can be compactified by adding finitely many points. We denote the compact surface by $X\left( \Gamma \right)$, and refer to the points of  $X\left( \Gamma \right) \setminus Y \left( \Gamma \right)$ as the cusps of $X\left( \Gamma \right)$. Note that $X\left( \Gamma \right)$ is a compact Riemann surface, and hence an algebraic curve, and the cusps are in fact orbits of $\mathbb{P}^{1} \left( \Q \right) = \Q \union \{ \infty \}$ under the natural extension of the action of $\Gamma$ to this set. 

Let $N$ be any positive integer. We recall the definitions of some standard congruence subgroups of $SL_{2}\left( \Z \right)$;
\begin{align*}
& \Gamma \left( N \right) = \{ \begin{psmallmatrix}
     a & b \\ c & d 
\end{psmallmatrix} \in SL_{2} \left( \Z \right) \ : \ a,d \  \equiv \  1 \  \text{mod} \ N   \ \text{and} \   b,c  \ \equiv  \ 0 \ \text{mod} \ N \}, \\
& \Gamma_{1} \left( N \right) = \{ \begin{psmallmatrix}
     a & b \\ c & d 
\end{psmallmatrix} \in SL_{2} \left( \Z \right) \ : \ a,d \  \equiv \  1 \  \text{mod} \ N   \ \text{and} \   c  \  \equiv \ 0 \ \text{mod} \ N \}, \\
& \Gamma_{0} \left( N \right) = \{ \begin{psmallmatrix}
     a & b \\ c & d \end{psmallmatrix} \in SL_{2} \left( \Z \right) \ : \ c  \  \equiv \ 0  \ \text{mod} \ N \}.
\end{align*}
The associated modular curves $X\left( \Gamma \left( N \right) \right)$,  $X\left( \Gamma_{1} \left( N \right) \right)$ and $X\left( \Gamma_{0} \left( N \right) \right)$ are denoted by $X\left( N \right)$, $X_{1}\left( N \right)$ and $X_{0} \left( N \right)$ respectively. It can be shown that all of these curves have models over $\Q$.

Let $H$ be any subgroup of the multiplicative group $\left( \mathbb{Z} / N \mathbb{Z} \right)^{*}$. Corresponding to $H$, one can define the following congruence subgroup
\begin{center}
    $\Gamma_{H}\left( N \right) = \{\begin{psmallmatrix} a & b \\c  & d \end{psmallmatrix} \in SL_{2}\left( \Z \right) \ : c \ \equiv \ 0 \ \text{mod} \ N, \ \ \  a,d \ \text{mod} \ N \in H  \}$.
\end{center}
We will write $X_{H}\left( N \right)$ for the modular curve associated to $\Gamma_{H}\left( N \right)$. The $X_{H}\left( N \right)$ has a model over a subfield of  $\Q \left( \zeta_{N} \right)$ depending on $H$.

\begin{remark}
As $-I$ acts the same as $I$ on $\mathbb{H}$, $X_{\langle \pm 1, H \rangle} \left( N \right) = X_{H}\left( N \right)$, so we always assume $-1 \in H $.
\end{remark}

The classical curves $X_{1}\left( N \right)$ and $X_{0}\left( N \right)$ can be recovered by taking $H = \{ \pm 1 \}$ and $H = \left( \Z / N \Z \right)^{*}$ respectively. 

The above definitions show that for any $N$ and $H$, we have the following inclusions:
\begin{center}
    $ \Gamma\left( N \right) \subseteq \pm \Gamma_{1}\left( N \right) \subseteq \Gamma_{H}\left( N \right) \subseteq  \Gamma_{0}\left( N \right) $
\end{center}
and these inclusions induce natural Galois coverings:
\begin{center}
  $X\left( N \right) \rightarrow X_{1}\left(N \right) \rightarrow X_{H}\left( N \right) \rightarrow X_{0}\left( N \right)$
\end{center}
where the above maps are simply quotient maps.
We may use the above to compute the genus of $X_{H}\left( N \right)$.

For any divisor $d$ of $N$, let $\pi_{d}$ be the projection 
\begin{center}
   $ \left( \Z / N \Z \right)^{*} \longrightarrow \left( \Z / \text{lcm}\left( d, N/d \right) \Z \right)^{*} $.
\end{center}

The following theorem is proved in \cite{JQdim}.
\begin{thm}
 The genus of $X_{H}\left(N \right)$ is 
 \begin{center}
 $g_{H}\left(N \right) = 1 + \frac{\mu \left( N, H \right)}{12} - \frac{\nu_{2} \left( N, H \right)}{4} - \frac{\nu_{3} \left( N, H \right)}{3} - \frac{\nu_{\infty} \left( N , H \right)}{2} $    
 \end{center}
 where 
 \begin{align*}
 & \mu \left(N , H \right) = N \displaystyle \prod_{p \mid N } \left( 1 + 1/p \right) \varphi \left( N \right) / \mid H \mid; \\
 & \nu_{2} \left( N , H \right) = \mid \{ b \ \text{mod} \ N \in H \ \mid \ b^{2} + 1 \ \equiv \ 0 \ \text{mod} \ N \} \mid\varphi \left( N \right) / \mid H \mid ;\\
 & \nu_{3} \left( N , H \right) = \mid \{ b \ \text{mod} \ N \in H \ \mid \ b^{2} - b  + 1 \ \equiv \ 0 \ \text{mod} \ N \} \mid\varphi \left( N \right) / \mid H \mid; \\
 & \nu_{\infty} \left( N , H \right) = \displaystyle \sum_{d \mid N} \frac{\varphi\left( d \right)\varphi \left( N/d \right)} {\pi_{d} \left( H \right)};
 \end{align*}
 and $\varphi$ denotes the usual Euler totient function.
\end{thm}

From now on, we assume $N=p$ is prime, $ p \geq 5$ and $H$ is the subgroup of non-zero squares modulo $p$. We additionally assume that $p$ is congruent to $1$ modulo $4$, in order to ensure that $-1 \in H $. Note that such curves have models over $\Q$.

\begin{proposition}
With $p$ and $H$ as above, the genus $g_{H}\left( p \right) $ of $X_{H}\left( p \right)$ is:
\begin{equation*}
  g_{H}\left( p \right) =
    \begin{cases}
      \frac{p-5}{6} & \text{if} \  p \ \equiv \ 5 \ \text{mod} \ 24 \\
       \frac{p - 1 }{6}  -3 & \text{if} \  p \ \equiv \ 1 \ \text{mod} \ 24 \\
     \frac{p -17}{6} + 1 & \text{if} \ p \ \equiv \ 17 \ \text{mod} \ 24 \\
    \frac{p-13}{6} & \text{if} \  p \ \equiv \ 13 \ \text{mod} \ 24 
      \end{cases}       
\end{equation*}
\end{proposition}
\begin{proof}
Applying the theorem above in our case, it is an elementary exercise to prove 
\begin{equation*}
    \nu_{2}\left(p, H \right) = \begin{cases}
        4 & \text{if} \ p \equiv 1 \ \text{mod} \ 8 \\
        0 & \text{otherwise}
    \end{cases}
\end{equation*}
\begin{equation*}
    \nu_{3}\left( p , H \right) = \begin{cases}
        4 & \text{if} \ p \ \equiv \ 1 \ \text{mod} \ 12 \\
        2 & \text{if} \ p \ \equiv \ 7 \ \text{mod} \ 12 \\
        0 & \text{otherwise}
    \end{cases}
\end{equation*}
\begin{equation*}
    \nu_{\infty}\left(p, H \right) = \begin{cases}
        4 & \text{if} \ p \equiv 1 \ \text{mod} \ 4 \\
        2 & \text{if} \ p \equiv 3 \ \text{mod} \ 4 \\    \end{cases}
\end{equation*}
\end{proof}

The cusps of all modular curves considered can be described abstractly. We will only consider the case $N=p$ prime and drop the standing assumption that $p$ is congruent to $1$ modulo 4 for our initial description. For arbitrary level $N$, see Ogg's classical description of cusps \cite{Ogg}.

A cusp of $X\left( p \right)$ can be represented by a vector $\pm \begin{psmallmatrix}
    x \\ y 
\end{psmallmatrix}$ with $x,y \in \Z  /p \Z $ and $\text{gcd}\left( x, y , p \right) =1$. 
The curve is defined over $\Q$ and the cusps are rational over $\Q \left( \zeta_{p} \right)$, with the Galois group $\left( \Z/p\Z \right)^{*}$ operating as $\begin{psmallmatrix}
  1 & 0 \\ 0 & \sigma
\end{psmallmatrix}$, where $\sigma \in \left( \Z / p \Z \right)^{*}$ represent the automorphism $\zeta_{p} \mapsto \zeta_{p}^{\sigma}$, hence 
\begin{center}
    $\sigma \begin{psmallmatrix}
        x \\ y 
    \end{psmallmatrix} = \begin{psmallmatrix}
        1 & 0 \\ 0 & \sigma 
    \end{psmallmatrix} \begin{psmallmatrix}
        x \\ y 
    \end{psmallmatrix} = \begin{psmallmatrix}
         x \\ \sigma y
    \end{psmallmatrix}.$
\end{center}

The cusps of $X_{1}\left( p \right)$ are orbit classes under the action of $\begin{psmallmatrix}
    1 & 1 \\ 0 & 1 
\end{psmallmatrix}$, or equivalently, they are equivalence classes with respect to the relation
\begin{center}
    $\pm \begin{psmallmatrix}
        x \\ y 
    \end{psmallmatrix} \sim  \pm \begin{psmallmatrix}
        x + ry \\ y 
    \end{psmallmatrix}$ for $r \in \Z. $
\end{center}
For each class, we may take a normalized representative $\begin{psmallmatrix}
    x \\ y 
\end{psmallmatrix}$ with 
\begin{itemize}
    \item $0 \leq y < p$,
    \item $0 \leq x <\text{gcd}\left( p , y \right)$.
\end{itemize}
This splits the cusps of $X_{1}\left( p \right)$ into two natural groups, which can be parameterized using any generator $\alpha$ of $\left( \Z / p \Z \right)^{*}/ \pm 1$:
\begin{itemize}
\small
    \item[(C1)]  Cusps of the form $\begin{psmallmatrix}
        \alpha^{k} \\ 0 
    \end{psmallmatrix}$ with  $k = 0, \ldots, \frac{p-1}{2}$;
    \vspace{0.25cm}
    \small 
      \item[(C2)]  Cusps of the form $ \begin{psmallmatrix} 0\\
        \alpha^{k} 
    \end{psmallmatrix} $ with $k = 0, \ldots,\frac{p-1}{2}$. \end{itemize}  

The cusps of type (C1) are all rational, and the cusps of type (C2) form a single orbit under the action of $\Gal \left( \Q \left( \zeta_{p} \right)^{+}/ \Q \right)$.

The cusps of $X_{0}\left( p \right)$ are orbits of $G_{0}\left( p \right) = \Gamma_{0}\left( p \right) / \Gamma_{1}\left( p \right)$. We find that all cusps of type (C1) form a single orbit under this action, and hence they map to a single cusp $c_{\infty} \in X_{0}\left( p \right)$, represented by $\begin{psmallmatrix}
    1 \\ 0 
\end{psmallmatrix}$, which will be referred to as the cusp at infinity, and similarly, all cusps of type (C2) are identified by the above action and hence all map to a single cusp $c_{0} \in X_{0}\left( p \right)$, represented by $\begin{psmallmatrix}
    0 \\ 1
\end{psmallmatrix}$, which we'll call the zero cusp. Note that both cusps are defined over $\Q$.

Suppose that $p$ is congruent to $1$ modulo $4$, and $H$ is the subgroup of non-zero squares modulo $p$, as before. The cusps of $X_{H}\left( p \right)$ are orbits of $G_{H}\left( p \right) = \Gamma_{H}\left( p \right) / \Gamma_{1}\left( p \right), $ and we find that cusps of type (C1) form two orbits under this action, 
\begin{center}
    $\begin{psmallmatrix}
        1 \\ 0 
    \end{psmallmatrix}, 
\begin{psmallmatrix}
        \alpha^{2} \\ 0 
    \end{psmallmatrix}, \ldots,\begin{psmallmatrix}
        \alpha^{s-2} \\ 0 
    \end{psmallmatrix}$ and 
     $\begin{psmallmatrix}
        \alpha \\ 0 
    \end{psmallmatrix}, 
\begin{psmallmatrix}
        \alpha^{3} \\ 0 
    \end{psmallmatrix}, \ldots,\begin{psmallmatrix}
        \alpha^{s-1} \\ 0 
    \end{psmallmatrix};$ 
    \end{center}
and as do the cusps of type (C2):
\begin{center}
    $\begin{psmallmatrix}
        0 \\ 1 
    \end{psmallmatrix}, 
\begin{psmallmatrix}
       0\\ \alpha^{2} 
    \end{psmallmatrix}, \ldots,\begin{psmallmatrix}
    0 \\    \alpha^{s-2} 
    \end{psmallmatrix}$ and 
     $\begin{psmallmatrix}
      0 \\   \alpha 
    \end{psmallmatrix}, 
\begin{psmallmatrix}
     0 \\   \alpha^{3}  
    \end{psmallmatrix}, \ldots,\begin{psmallmatrix}
      0 \\   \alpha^{s-1} 
    \end{psmallmatrix}$ 
    \end{center}
Thus $X_{H}\left( p \right) $ has four cusps $c_{1}, c_{2}, c_{3}, c_{4}$, represented by $\begin{psmallmatrix}
        0 \\ 1 
    \end{psmallmatrix}, 
\begin{psmallmatrix}
        0 \\ \alpha 
    \end{psmallmatrix},  
    \begin{psmallmatrix}
        1 \\  0 
    \end{psmallmatrix}$  and $\begin{psmallmatrix}
        \alpha  \\ 0 
    \end{psmallmatrix}, $  
    respectively. The cusps $c_{1}$ and $c_{2}$ are defined over $\Q$, whilst $c_{3}$ and $c_{4}$ are defined over $\Q \left( \sqrt{p} \right)$, the unique degree 2 subfield of $\Q \left( \zeta_{p} \right)$, and they are Galois conjugates.
      
    \section{Computational Overview} 
Fix a prime $p$, congruent to $1$ modulo $4$. Let $H$ be the subgroup of $\left( \Z / p \Z \right)^{*} $ 
consisting of squares and consider the corresponding modular curve $X_{H}\left( p \right)$. Firstly, using the genus formulas given in the previous sections, we find that there are precisely $6$ values of $p$, as above, for which the modular curve $X_{H}\left( p \right)$ has genus $g_{H} \left( p \right)$ such that $2 \leq g_{H} \left( p \right) \leq 10$ : 
\begin{itemize}
\item $ g_{H}\left( p \right) = 4 : p = 29, 37 $;
\item $ g_{H}\left( p \right) = 5 : p = 41 $;
\item $g_{H}\left( p \right) = 8 : p = 53, 61 $;
\item $ g_{H}\left( p \right) = 9 : p = 73 $.
\end{itemize}

We begin our computations by computing models for the modular curves $X_{0}\left( p \right)$ and $X_{H}\left( p \right)$ for our choices of $p$. This was done using Galbraiths's method \cite{GalMet} which is summarised in the following subsection. A concise summary of this method can also be found in \cite{siksek2019explicit}.

\subsection{Models for Modular Curves: Galbraith's Method}
Let $X$ be a modular curve of genus $g \geq 2 $, corresponding to a congruence subgroup $\Gamma$. The space of holomorphic differentials $\Omega^{1}\left( X \right) $ and the space of weight 2 cups forms $S_{2}\left( \Gamma \right)$ on $\Gamma$ are isomorphic as complex vector spaces. More explicitly, if $\{ f_{1}\left( \tau \right), \ldots, f_{g}\left( \tau \right)\}$ is a basis for $S_{2}\left( \Gamma \right)$, then $\{ f_{1}\left( \tau \right) d\tau , \ldots, f_{g}\left( \tau \right) d\tau \}$ is a basis for $\Omega^{1} \left( X \right) $. Thus, the canonical map is simply:
\begin{center}
  $  \phi : X  \longrightarrow \mathbb{P}^{g-1} $, \\
  $ \phi \left( \tau \right) = \left( f_{1}\left( \tau \right) :  \ldots : f_{g}\left( \tau \right) \right) $.
\end{center}
The canonical map is an embedding if and only if the curve is not hyperelliptic. If $X$ is hyperelliptic, the image of the canonical map is isomorphic to $\mathbb{P}^{g-1}$ and it is described by $1/2 \left( g-1 \right) \left( g - 2 \right)$ quadrics. Thus we can determine whether the curve is hyperelliptic or not by computing the image of the canonical map. 

If $X$ is non-hyperelliptic, the image of the canonical map is a curve of degree $2g-2$ and it will be described by a set of projective equations of the form $\Phi \left( f_{1}, \ldots, f_{g} \right) =0$. In the case that it is a complete intersection, the equations have degrees whose product is $2g-2$. We interpret each $\Phi \left( f_{1}, \ldots, f_{g} \right) =0$ as a modular form which vanishes on the extended upper half plane, and hence these forms may be interpreted as the zero cusps form of weight $2 \cdot \text{deg}\left( \Phi \right)$. This gives a strategy for finding a model of the curve.

To find a defining set of equations for a non-hyperelliptic modular curve, take a basis $ f_{1},$ $\ldots$ $ , f_{g}$ for $S_{2}\left( \Gamma \right)$. For any degree $d$, homogeneous polynomial $F \in \Q \left[ x_{1}, \ldots, x_{g} \right]$, $F\left( f_{1}, \ldots, f_{g} \right)$ is a cusps form of weight $2d$ and level $N$. Let $I = \left[ \text{SL}_{2}\left( \Z \right) : \Gamma \right] = N \displaystyle \prod_{p \vert N} \left( 1 + 1/p \right)$.
\begin{thm}{(Sturm's theorem)}
With notation as above, $F\left( f_{1}, \ldots, f_{g} \right) = 0 $ if and only if the $q$ expansion satisfies $F\left( f_{1}, \ldots, f_{g} \right) = O\left( q^{r} \right)$, where $r = \lfloor dI/6 \rfloor + 1  $.
\end{thm}

Using the above result, we have reduced our problem to an  elementary linear algebra problem, namely that of computing the vector space of all homogeneous $F$ of degree $d$ such that $F\left( f_{1}, \ldots, f_{g} \right) = 0$. This can be repeated for all divisors $d$ of  $2g -2$, resulting in a system of equations that cuts out a model for $X$ in $\mathbb{P}^{g-1}$.

Suppose $X$ is a hyperelliptic modular curve of genus $g \geq 2$. As before, there is an isomorphism $S_{2}\left( \Gamma \right) \cong \Omega \left( X \right)$. Let  $ f_{1},$ $\ldots$ $ , f_{g}$ be a basis for $S_{2}\left( \Gamma \right)$, and since $X$ is modular, we can choose a basis such that $f_{i}$ has $q$-expansion $ q^{i} + \ldots $ . As $X$ is hyperelliptic, it has an affine model of the form  
\begin{center}
    $X : y^{2} = f\left( x \right)$ 
\end{center}
with $f\in \Q \left[ x \right]$, of degree $2g+1$ or $2g +2$. Then $\Omega \left( X \right)$ has a basis $dx/y, xdx/y, \ldots, x^{g-1}dx/y$. Note that this basis is not necessarily the same as the basis $f_{1}dq/q, \ldots, f_{g}dq/q$. Changing coordinates if necessary, we can assume that the infinity cusp $c_{\infty}$ of $X$, the class of $\infty \in \mathbb{H}^{*}$ in the quotient interpretation, is one of the points at infinity of the hyperelliptic curve in our affine model. Then $q$ is a uniformizer at $c_{\infty}$. To differentiate between the cases when $f$ has degree $2g+2$ or $2g+1$, we use the following result.

\begin{proposition}
Suppose $X$ is a hyperelliptic curve of genus $g \geq 2$, and has an affine model 
\begin{center}
$X : y^{2} = a_{2g+2}x^{2g+2} + \ldots + a_{0}$ with $a_{2g+2} \neq 0$
\end{center}
and let $\infty_{+}$ be one of the points at infinity. Then 
\begin{center}
   $\text{ord}_{\infty_{+}}\left( dx/y \right) = g-1 $, $\text{ord}_{\infty_{+}}\left( xdx/y \right) = g-2 $, $\ldots$ ,  $\text{ord}_{\infty_{+}}\left( x^{g-1}dx/y \right) = 0 $
\end{center}

If $X$ has affine model 
\begin{center}
$X : y^{2} = a_{2g+1}x^{2g+1} + \ldots + a_{0}$ with $a_{2g+1} \neq 0$
\end{center}
and $\infty$ is the unique point at infinity on the above model. Then 
\begin{center}
   $\text{ord}_{\infty} \left( dx/y \right) = 2(g-1) $, $\text{ord}_{\infty}\left( xdx/y \right) = 2(g-2) $, $\ldots$ ,  $\text{ord}_{\infty} \left( x^{g-1}dx/y \right) = 0 $
\end{center}
\end{proposition}

Using the $q$-expansions of $f_{1}, \ldots, f_{g}$, we determine if any linear combination of $f_{1}dq/q, \ldots, f_{g}dq/q$ has order $2(g-1)$ at $c_{\infty}$. If such a linear combination exists, then we have a degree $2g+1$ model; otherwise the polynomial in our affine model has degree $2g+2$. 

From the above valuations and the $q$-expansions of $f_{1}, \ldots, f_{g}$, we have 
\begin{align*}
& dx/y = \alpha_{g} f_{g}\left( q \right) dq/q \\
& xdx/y = \beta_{g-1}f_{g-1}\left( q \right) dq/q + \beta_{g}f_{g}\left( q \right) dq/q
\end{align*}
for some constants $\alpha_{g} \neq 0$, $\beta_{g-1} \neq 0$ and $\beta_{g}$. The change of coordinates 
\begin{center}
  $ x \mapsto rx, y \mapsto sy$
\end{center}
fixes the points at infinity, and scales the differentials as follows,
\begin{align*}
 & dx/y \mapsto \left( r/s \right) dx/y \\
 & xdx/y \mapsto \left( r^{2}/s \right) xdx/y 
 \end{align*}
 and thus we may take $\alpha_{g} = \beta_{g-1} =1 $. The change of coordinates 
 \begin{center}
  $ x \mapsto x + t , y \mapsto y$
 \end{center}
 also fixes the points at infinity, and transforms the differentials as follows
 \begin{align*}
 & dx/y \mapsto  dx/y \\
 & xdx/y \mapsto  xdx/y + tdx/y 
 \end{align*}
 and hence we may take $\beta_{g} = 0$. Thus we can assume 
 \begin{align*}
 & dx/y  =  f_{g}\left( q \right) dq/q \\
 & xdx/y = f_{g-1}\left( q \right) dq/q
 \end{align*}
and hence 
\begin{center}
$ x = f_{g-1}\left( q \right) / f_{g}\left( q \right) $  and $y = \left( dx/dq \right) \left( q/f_{g}\left( q \right)\right)$.
\end{center}
As in the previous case, we have reduced the problem to an elementary linear algebra exercise, as we now use these expressions to search for a polynomial $f \in \Q \left[ x \right]$ of degree $2g+1$ or $2g+2$, such that 
\begin{center}
   $y^{2} - f\left( x \right) = 0$ 
\end{center}
using the $q$-expansions of $f_{g}\left( q \right)$ and $f_{g-1}\left( q \right)$; and Sturm's theorem.

\subsection{The Quotient Map $X_{H}\left( p \right) \longrightarrow X_{0}\left( p \right)$}

Suppose both curves $X_{H}\left(p \right)$ and $X_{0}\left( p \right)$ are not hyperelliptic. This is the case for $p = 53,61$ and $73$. Let $f_{1}, \ldots, f_{g}$ be a basis of $S_{2}\left( \Gamma_{0}\left( p \right) \right)$. This can be extended to a basis $f_{1}, \ldots, f_{g}, f_{g+1} \ldots, f_{g_{H}}$ of  $S_{2}\left( \Gamma_{H}\left( p \right) \right)$. We compute models for the two curves with respect to the above bases and with coordinates $\left( x_{1}: \ldots :x_{g_{H}} \right) =\left( f_{1}: \ldots : f_{g_{H}} \right)$. Then the degree $2$ quotient map is simply the projection:
\begin{center}
   $ \varphi : X_{H}\left( p \right) \longrightarrow X_{0}\left( p \right) $ \\
   $  \varphi \left( x_{1}, \ldots, x_{g_{H}} \right) = \left( x_{1}, \ldots, x_{g} \right) $.
\end{center}

In the cases $p= 29, 37$ and $41$, the curve $X_{0}\left( p \right)$ is hyperelliptic, whilst $X_{H}\left( p\right)$ is not. Take a basis  $ f_{1},$ $\ldots$ $ , f_{g}$ for $S_{2}\left( \Gamma_{0}\left( N\right) \right)$, such that $f_{i}$ has $q$-expansion $ q^{i} + \ldots $. As described in the previous subsection, we find an affine  model of $X_{0}\left( p \right)$ 
\begin{center}
    $X_{0}\left( p \right): y ^{2} = f\left( x \right)$
\end{center}
where $f \in \Q \left[ x \right]$ has degree $2g+1$ or $2g +2$, and 
\begin{center}
$ x = f_{g-1}\left( q \right) / f_{g}\left( q \right) $  and $y = \left( dx/dq \right) \left( q/f_{g}\left( q \right) \right)$.
\end{center}
The basis  $ f_{1},$ $\ldots$ $ , f_{g}$ can be extended to a basis of $S_{2}\left( \Gamma_{H} \left( p \right) \right)$, $ f_{1},$ $\ldots$ $ , f_{g_{H}}$, and we compute the canonical model of the non-hyperelliptic curve $X_{H}\left( p \right)$ with respect to this basis.  The model of $X_{H}\left( p \right)$ will have coordinates
\begin{center}
$\left( x_{1} :  \ldots : x_{g-1} :  x_{g}:  \ldots :  x_{g_{H}} \right) = \left( f_{1} :  \ldots : f_{g-1} :  f_{g}:  \ldots :  f_{g_{H}} \right)$
\end{center}
and the degree $2$ map $ \varphi : X_{H}\left( p \right) \longrightarrow X_{0}\left( p \right)$ is simply 
\begin{center}
  $\left( x_{1} :  \ldots : x_{g-1} :  x_{g}:  \ldots :  x_{g_{H}} \right) \mapsto \left( x \left( x_{g-1}, x_{g} \right), y \left( x_{g-1}, x_{g} \right) \right)$  
\end{center}
where $x$ and $y$ are given by 
\begin{center}
$ x\left( x_{g-1}, x_{g} \right)  = x_{g-1} /x_{g} $  and $y = \sqrt{f \left( x\left( x_{g-1}, x_{g} \right) \right) }$.
\end{center}

\subsection{The Cuspidal Subgroup $C_{H}\left( p \right)$}
The cusps of $X_{H}\left( p \right)$, $c_{1}, \ldots, c_{4}$ are the inverse images of the cusps of $X_{0}\left( p \right)$ under the map $\varphi$. Then $C_{H}\left( p \right) $ is the subgroup generated by these cusps 
\begin{center}
    $C_{H} \left( p \right) = \langle \left[ c_{2} - c_{1} \right], \left[ c_{3} - c_{1} \right] , \left[ c_{4} - c_{1} \right] \rangle$.
\end{center}
This is a subgroup of $J_{H}\left( \Q\left( \sqrt{p } \right)\right)_{\text{tors}}$ by the theorems of Manin \cite{Manin} and Drinfeld \cite{Drinfeld}. To find a presentation of this group, and any relations amongst the cuspidal divisors, we reduce modulo a prime of good reduction in $\Q \left( \sqrt{p} \right)$ and use the following result. 

\begin{lemma}
Let $S = J_{H}\left( p \right) \left( \Q \left( \sqrt{p} \right) \right)$, and let $\mathfrak{q}$ be any prime in $\Q \left( \sqrt{p} \right)$, not lying above $p$ or $2$. Then reduction modulo $\mathfrak{q}$ induces an injection 
\begin{center}
    $S \longrightarrow J_{H}\left( \mathbb{F}_{\mathfrak{q}} \right)$
\end{center}
\end{lemma}
\begin{proof}
See \cite{katz}.
\end{proof}

By our choice of model (and notation), $c_{3}$ and $c_{4}$ are defined over $\Q \left( \sqrt{p} \right)$ and Galois conjugate, and $c_{1}$, $c_{2}$ are rational. Thus the Galois action on $C_{H}\left( p \right)$ is clear, and by   taking Galois invariants we obtain the rational cuspidal subgroup $C_{H}\left( p \right) \left( \Q \right)$.

\subsection{Torsion Subgroup of $J_{H}\left( p \right)$}
For any prime ideal $\mathfrak{q}$ of $\mathcal{O}$, the ring of integers of $\Q \left( \sqrt{p} \right)$, whose norm is coprime to $p$, reduction modulo $\mathfrak{q}$ induces an injection 
\begin{center}
    $J_{H}\left( p \right) \left( \Q \left( \sqrt{p} \right) \right)_{\text{tors}}  \rightarrow J_{H}\left( \mathbb{F}_{\mathfrak{q}} \right)$.
\end{center}
Reducing modulo multiple prime ideals, gives an upper bound on the size of  $J_{H}\left( p \right) \left( \Q \left( \sqrt{p} \right)  \right)_{\text{tors}} $, and in our computations, we found sufficiently many primes such that this bound matched to size of $C_{H}\left( p \right)$, and hence we were able to conclude $ J_{H}\left( p \right) \left( \Q \left( \sqrt{p} \right) \right)_{\text{tors}} = C_{H} \left( p \right)$.

\section{Proof of the Main Theorem}

The \texttt{MAGMA} code used to carry out the computations presented in this section can be found in the following online repository:
\begin{center}
\href{ https://github.com/ElviraLupoian/XHSquares}{https://github.com/ElviraLupoian/XHSquares}
\end{center}
\subsection{ Curves of Genus 4}
There are 2 primes for which $X_{H}\left( p \right)$ has genus 4, namely $29$ and $37$. Both curves $X_{H}\left(p\right)$ are not hyperelliptic, so their canonical model is the intersection of a quadric and a cubic in $\mathbb{P}^{3}$, and in both cases $X_{0}\left( p \right)$ has genus 2.
\subsubsection*{p = 29}
A canonical model of $X_{H}\left( 29 \right)$ is given by the zero locus of 
\begin{center}
$ c = x_{1}^2x_{3} - 3x_{1}x_{2}x_{3} - x_{1}x_{3}^2 - x_{2}^3 + x_{2}^2x_{3} + 
    x_{2}x_{3}^2 + 5x_{2}x_{3}x_{4} - x_{2}x_{4}^2 + 4x_{3}^3 - 5x_{3}^2x_{4} - 
    3x_{3}x_{4}^2; $ \\
$ q = x_{1}x_{4} - x_{2}x_{3} + x_{2}x_{4} + x_{3}^2 - 3x_{3}x_{4} - 2x_{4}^2.$
\end{center}
The cusps of $X_{H}\left( 29 \right)$ are
\begin{align*}
    c_{1} & = \left( 1: 0 : 0 : 0 \right),  \ c_{2} = \left( 2 : 0 : 0 : 1 \right); \\
    c_{3} & = \left( 3 \sqrt{29} + 18 :  \sqrt{29} + 5 : 1/2\sqrt{29} + 5/2 : 1 \right); \\  
    c_{4}  & = \left( -3 \sqrt{29} + 18 :   - \sqrt{29} + 5 : - 1/2\sqrt{29} + 5/2 : 1 \right).
    \end{align*}
The cuspidal subgroup $C_{H}\left( 29 \right) = \langle \left[ c_{2} - c_{1} \right] ,\left[ c_{3} - c_{1} \right],\left[ c_{4} - c_{1} \right] \rangle$ is isomorphic to $ \left( \mathbb{Z} / 3 \mathbb{Z} \right) \times \left( \mathbb{Z} /21 \mathbb{Z} \right) $, and it fact, $ \left[ c_{3} - c_{1} \right] $ and $ \left[ c_{4} - c_{1} \right]$ are sufficient to generate the entire group. Taking Galois invariants, the resulting rational cuspidal subgroup is 
\begin{center}
$C_{H}\left( 29 \right)  \left( \mathbb{Q} \right) = \langle \left[ c_{3} + c_{4} - 2c_{1} \right] \rangle \cong  \left( \mathbb{Z} /21 \mathbb{Z} \right) $ 
\end{center}
The upper bound for  $\vert J_{H} \left( 29 \right) \left( \mathbb{Q} \left( \sqrt{29} \right)\right)_{\textbf{tors}} \vert $ resulting from reduction modulo multiple primes is exactly $63$, and thus we conclude 
\begin{center}
    $ J_{H} \left( 29 \right) \left( \mathbb{Q} \left( \sqrt{29} \right)\right)_{\textbf{tors}} = C_{H}\left( 29 \right) $.
\end{center}

\subsubsection*{p = 37}
A canonical model of $X_{H}\left( 37 \right)$ is given by the zero locus of 
\begin{align*}
 c &  = 2x_{1}^2x_{4} - 5x_{1}x_{4}^2 - 2x_{2}^3 + 2x_{2}^2x_{3} - 2x_{2}x_{3}^2 + 
    6x_{2}x_{3}x_{4} - 6x_{2}x_{4}^2 - 3x_{3}^3 + 8x_{3}^2x_{4} - 9x_{3}x_{4}^2
    + 6x_{4}^3; \\
 q & = x_{1}x_{3} - x_{2}^2 - 2x_{3}x_{4}. 
\end{align*}
The cusps of $X_{H}\left( 37 \right)$ are
\begin{align*}
    c_{1} & = \left( 1: 0 : 0 : 0 \right),  \ c_{2} = \left( 2 : 0 : 1 : 1 \right); \\
    c_{3} & = \left(  \sqrt{37} + 5 : 6 :  \sqrt{37 } -1 : 2 \right);   \\  
   c_{4} & = \left(  -\sqrt{37} + 5 : 6 :  -\sqrt{37 } -1 : 2 \right). 
      \end{align*}
The cuspidal subgroup $C_{H}\left( 37 \right) = \langle \left[ c_{2} - c_{1} \right] ,\left[ c_{3} - c_{1} \right],\left[ c_{4} - c_{1} \right] \rangle$ is isomorphic to $ \left( \mathbb{Z} / 5 \mathbb{Z} \right) \times \left( \mathbb{Z} /15 \mathbb{Z} \right) $, and it fact, $ \left[ c_{3} - c_{1} \right] $ and $ \left[ c_{4} - c_{1} \right]$ are sufficient to generate the entire group. Taking Galois invariants, the resulting rational cuspidal subgroup is 
\begin{center}
$C_{H}\left( 37 \right)  \left( \mathbb{Q} \right) = \langle \left[ c_{3} + c_{4} - 2c_{1} \right] \rangle \cong  \left( \mathbb{Z} /15 \mathbb{Z} \right). $ 
\end{center}

The upper bound for  $\vert J_{H} \left( 37  \right) \left( \mathbb{Q} \left( \sqrt{37} \right)\right)_{\textbf{tors}} \vert $ obtained from reducing modulo a number of primes is exactly 75, and hence the two groups must be equal;
\begin{center}
    $ J_{H} \left( 37 \right) \left( \mathbb{Q} \left( \sqrt{37} \right)\right)_{\textbf{tors}} = C_{H}\left( 37 \right) \cong \Z / 5 \Z \times \Z / 15 \Z $.
\end{center}

\subsection{ Curves of Genus 5}
There is a single prime congruent to $1$ modulo $4$ for which $X_{H}\left( p \right)$ has genus 5, namely $41$. The curve $X_{H}\left(41\right)$ is not hyperelliptic, so its canonical model is the intersection of $3$ quadrics in  $\mathbb{P}^{4}$. 

A canonical model of $X_{H}\left( 41 \right)$ is given by the zero locus of 
\begin{align*}
Q_{1} & = x_{1}x_{3} - x_{2}^2 + 2x_{2}x_{4} + 2x_{2}x_{5} - 2x_{3}^2 - x_{3}x_{4} + 
    x_{3}x_{5} - 2x_{4}^2 - 2x_{4}x_{5} - x_{5}^2; \\
Q_{2} &= x_{1}x_{4} - x_{2}x_{3} + x_{2}x_{4} - x_{3}^2 + x_{3}x_{4} + 2x_{3}x_{5} - 2x_{4}^2 
    - 2x_{4}x_{5}; \\
Q_{3} & = x_{1}x_{5} - x_{2}x_{5} - x_{3}^2 + 2x_{3}x_{4} + 2x_{3}x_{5} - 2x_{4}^2 - 
    x_{4}x_{5} + x_{5}^2.
\end{align*}

The cusps of $X_{H}\left( 41 \right)$ are
\begin{align*}
    c_{1} & = \left( 1: 0 : 0 : 0 : 0  \right),  \ c_{2} = \left( 0 : 1 : 0 : 0 : 1 \right); \\
    c_{3} & = \left(18\sqrt{41} + 114 : 6\sqrt{41} + 54 : 6\sqrt{41} + 42 : 3\sqrt{41} + 21 : 12  \right);   \\  
   c_{4} & = \left(-18\sqrt{41} + 114 : -6\sqrt{41} + 54 : -6\sqrt{41} + 42 : -3\sqrt{41} + 21 : 12  \right). 
   \end{align*}
The cuspidal subgroup $C_{H}\left( 41 \right) = \langle \left[ c_{2} - c_{1} \right] ,\left[ c_{3} - c_{1} \right],\left[ c_{4} - c_{1} \right] \rangle$ is isomorphic to $ \left( \mathbb{Z} / 8 \mathbb{Z} \right) \times \left( \mathbb{Z} /40 \mathbb{Z} \right) $, and it fact, $ \left[ c_{3} - c_{1} \right] $ and $ \left[ c_{4} - c_{1} \right]$ are sufficient to generate the entire group. Taking Galois invariants, the resulting rational cuspidal subgroup is 
\begin{center}
$C_{H}\left( 41 \right)  \left( \mathbb{Q} \right) = \langle \left[ c_{3} + c_{4} - 2c_{1} \right] \rangle \cong  \left( \mathbb{Z} /40 \mathbb{Z} \right). $ 
\end{center}
The upper bound for $\vert J_{H}\left( 41 \right) \left( \Q \left( \sqrt{41} \right) \right)_{\text{tors}} \vert$ obtained from reducing modulo multiple primes is $320$, and hence it must equal the cuspidal subgroup;
\begin{center}
    $C_{H}\left( 41 \right) =J_{H}\left( 41 \right) \left( \Q \left( \sqrt{41} \right) \right)_{\text{tors}} \cong \Z / 8 \Z \times \Z / 40 \Z.$
\end{center}

\subsection{ Curves of Genus 8}
There are two primes congruent to $1$ modulo $4$ for which $X_{H}\left( p \right)$ has genus 8, namely $53$ and $61$. Both curves $X_{H}\left(p\right)$ are not hyperelliptic. 

\subsubsection*{p = 53}
A canonical model of $X_{H}\left( 53 \right)$ is given by the zero locus of the following fifteen quadrics and one cubic.
\allowdisplaybreaks
\begin{align*}
 Q_{1} & = 6x_{1}x_{3} - 6x_{2}^2 + 15x_{4}x_{6} + 5x_{4}x_{7} - 23x_{4}x_{8} - 15x_{5}^2 +
    33x_{5}x_{6} + 10x_{5}x_{7} + 60x_{5}x_{8} - 50x_{6}^2  \\ & - 54x_{6}x_{7} - 
    39x_{6}x_{8} - 18x_{7}^2 + 37x_{7}x_{8} + 15x_{8}^2; \\
Q_{2} & = 6x_{1}x_{4} - 6x_{2}x_{3}- 6x_{4}^2 - 43x_{4}x_{6} - 5x_{4}x_{7} - 17x_{4}x_{8} + 43x_{5}^2 - 13x_{5}x_{6} - 10x_{5}x_{7} \\ & - 226x_{5}x_{8}  - 32x_{6}^2 + 
    68x_{6}x_{7} + 205x_{6}x_{8} + 32x_{7}^2 + 9x_{7}x_{8}  + 225x_{8}^2; \\
Q_{3} & = 3x_{1}x_{5} - 3x_{3}^2 - 16x_{4}x_{6} - 28x_{4}x_{8} + 16x_{5}^2 - x_{5}x_{6} - 3x_{5}x_{7} - 100x_{5}x_{8} - 19x_{6}^2 + 23x_{6}x_{7}  \\ & + 73x_{6}x_{8} + 11x_{7}^2 + 13x_{7}x_{8} + 144x_{8}^2; \\ 
Q_{4} & = 3x_{1}x_{6} - 3x_{3}x_{4} + 3x_{4}^2 - 9x_{4}x_{6} - 2x_{4}x_{7} - 31x_{4}x_{8} + 12x_{5}^2 - 6x_{5}x_{6} + 2x_{5}x_{7} - 66x_{5}x_{8} \\ & - 16x_{6}^2  + 3x_{6}x_{7}  + 48x_{6}x_{8} + 6x_{7}^2 + 14x_{7}x_{8} + 114x_{8}^2; \\
Q_{5} & = 3x_{1}x_{7} - 3x_{4}x_{5} + 5x_{4}x_{6} - 3x_{4}x_{7} + 23x_{4}x_{8} - 5x_{5}^2 - x_{5}x_{6} + 38x_{5}x_{8} + 14x_{6}^2 + 2x_{6}x_{7} \\ & - 23x_{6}x_{8} - x_{7}^2- 17x_{7}x_{8} - 81x_{8}^2; \\
Q_{6} & = 3x_{1}x_{8} - 2x_{4}x_{6} - x_{4}x_{7} + 2x_{4}x_{8} - x_{5}^2 + 4x_{5}x_{6} + x_{5}x_{7} + x_{5}x_{8} - x_{6}^2  + x_{6}x_{7} - 4x_{6}x_{8} \\ & + x_{7}^2 - 6x_{7}x_{8} - 9x_{8}^2; \\
Q_{7} & = 3x_{2}x_{4} - 3x_{3}^2 - 3x_{4}x_{5} - 5x_{4}x_{6} - 3x_{4}x_{7} + x_{4}x_{8} + 5x_{5}^2 - 5x_{5}x_{6} + 3x_{5}x_{7} - 11x_{5}x_{8} \\ & + 10x_{6}^2  + 10x_{6}x_{7} + 5x_{6}x_{8} + 4x_{7}^2 - 10x_{7}x_{8} + 3x_{8}^2;\\
Q_{8} & = 3x_{2}x_{5} - 3x_{3}x_{4} - 11x_{4}x_{6} + x_{4}x_{7} - 3x_{4}x_{8} + 14x_{5}^2 -11x_{5}x_{6} - 7x_{5}x_{7} - 86x_{5}x_{8} + 3x_{6}^2 \\ & 
+ 34x_{6}x_{7} +  62x_{6}x_{8} + 13x_{7}^2 - 11x_{7}x_{8} + 75x_{8}^2; \\
    Q_{9} & = 3x_{2}x_{6} - 3x_{4}^2 + 3x_{4}x_{5} - 12x_{4}x_{6} + 4x_{4}x_{7} + 2x_{4}x_{8} 
    + 12x_{5}^2 - 6x_{5}x_{6} - 10x_{5}x_{7} - 75x_{5}x_{8}  \\ & - 4x_{6}^2 + 
    30x_{6}x_{7} + 57x_{6}x_{8} + 9x_{7}^2 - 4x_{7}x_{8} + 63x_{8}^2; \\
    Q_{10} & = 3x_{2}x_{7} + 3x_{4}x_{6} - x_{4}x_{7} + x_{4}x_{8} - 6x_{5}^2 + 3x_{5}x_{6} + 
    x_{5}x_{7} + 33x_{5}x_{8} - 2x_{6}^2 - 15x_{6}x_{7} \\ & - 21x_{6}x_{8}  - 6x_{7}^2
    + 4x_{7}x_{8} - 27x_{8}^2; \\
    Q_{11} & = 3x_{2}x_{8} + 2x_{4}x_{6} + 2x_{4}x_{8} - 2x_{5}^2 - x_{5}x_{6} + 11x_{5}x_{8} + 
    5x_{6}^2 - x_{6}x_{7} - 11x_{6}x_{8} - x_{7}^2 \\ & - 2x_{7}x_{8} - 18x_{8}^2;\\ 
    Q_{12} & = 6x_{3}x_{5} - 6x_{4}^2 + 5x_{4}x_{6} + 13x_{4}x_{7} + 43x_{4}x_{8} - 5x_{5}^2 + 
    17x_{5}x_{6} - 10x_{5}x_{7} - 16x_{5}x_{8} + 4x_{6}^2  \\ & + 26x_{6}x_{7}   - 
    17x_{6}x_{8} + 2x_{7}^2 - 51x_{7}x_{8} - 57x_{8}^2; \\
    Q_{13} & = 3x_{3}x_{6} - 3x_{4}x_{5} + 2x_{4}x_{6} + x_{4}x_{7} + 13x_{4}x_{8} + x_{5}^2 + 
    2x_{5}x_{6} + 2x_{5}x_{7} - 4x_{5}x_{8} + 4x_{6}^2 + 5x_{6}x_{7} - 
    \\ &  11x_{6}x_{8} + 2x_{7}^2 - 24x_{7}x_{8} - 21x_{8}^2; \\    
Q_{14} & = 3x_{3}x_{7} - 2x_{4}x_{6} - x_{4}x_{7} - 10x_{4}x_{8} + 2x_{5}^2 - 5x_{5}x_{6} - 
    2x_{5}x_{7} - 8x_{5}x_{8} - 4x_{6}^2 - 5x_{6}x_{7} \\ & + 14x_{6}x_{8}   - 2x_{7}^2
    + 15x_{7}x_{8} + 30x_{8}^2; \\
Q_{15} & = 6x_{3}x_{8} + 5x_{4}x_{6} + x_{4}x_{7} + x_{4}x_{8} - 5x_{5}^2 + 5x_{5}x_{6} + 
    2x_{5}x_{7} + 20x_{5}x_{8} - 2x_{6}^2 - 10x_{6}x_{7} \\ & - 23x_{6}x_{8}   - 
    4x_{7}^2 - 3x_{7}x_{8} - 21x_{8}^2;\\
c & = 138x_{1}^2x_{3} - 138x_{1}x_{2}^2 + 414x_{4}^3 - 1794x_{4}^2x_{5} + 
    2346x_{4}x_{5}^2   + 636x_{4}x_{7}^2 - 7278x_{4}x_{7}x_{8} \\ &  - 5148x_{4}x_{8}^2 
    - 2346x_{5}^3   + 2898x_{5}^2x_{6} + 1978x_{5}^2x_{7} + 2291x_{5}^2x_{8} + 
    3036x_{5}x_{6}^2 \\ & + 3643x_{5}x_{6}x_{7} - 16010x_{5}x_{6}x_{8}  + 
    1548x_{5}x_{7}^2 - 12320x_{5}x_{7}x_{8} + 31987x_{5}x_{8}^2 - 5152x_{6}^3 \\ &  - 
    9613x_{6}^2x_{7} - 3888x_{6}^2x_{8} - 1297x_{6}x_{7}^2  - 
    20654x_{6}x_{7}x_{8} - 3220x_{6}x_{8}^2 + 1139x_{7}^3 \\ & - 14602x_{7}^2x_{8} + 
    46267x_{7}x_{8}^2 - 15816x_{8}^3.
    \end{align*}

The cusps of $X_{H}\left( 53 \right)$ are
\begin{align*}
    c_{1} & = \left(1 : 0 : 0 : 0 : 0 : 0 : 0 : 0 \right); \\
    c_{2}  & = \left(1 : 2 : 1 : 1 : 2 : 1 : -1 : 1\right); \\
    c_{3} & = \left(  6\sqrt{53} + 46 : 4\sqrt{53} + 34 : 3\sqrt{53} + 25 : 3\sqrt{53} + 23 :  \sqrt{53}+ 13 :  2 :  \sqrt{53} + 5 :  2     \right); \\  
    c_{4}  & = \left(  -6\sqrt{53} + 46 : -4\sqrt{53} + 34 : -3\sqrt{53} + 25 : -3\sqrt{53} + 23 :  -\sqrt{53}+ 13 :  2 :  -\sqrt{53} + 5 :  2     \right).     \end{align*}
The cuspidal subgroup $C_{H}\left( 53 \right) = \langle \left[ c_{2} - c_{1} \right] ,\left[ c_{3} - c_{1} \right],\left[ c_{4} - c_{1} \right] \rangle$ is isomorphic to $ \left( \mathbb{Z} / 7 \mathbb{Z} \right) \times \left( \mathbb{Z} /91 \mathbb{Z} \right) $, and it fact, $ \left[ c_{3} - c_{1} \right] $ and $ \left[ c_{4} - c_{1} \right]$ are sufficient to generate the entire group. Taking Galois invariants, the resulting rational cuspidal subgroup is 
\begin{center}
$C_{H}\left( 53 \right)  \left( \mathbb{Q} \right) = \langle \left[ c_{3} + c_{4} - 2c_{1} \right] \rangle \cong  \left( \mathbb{Z} /91 \mathbb{Z} \right). $ 
\end{center}

Reducing modulo multiple primes, we find that $637$ is an upper bound for $\vert J_{H}\left( 53 \right) \left( \Q \left(  \sqrt{53} \right) \right)_{\text{tors}} $ and hence 
\begin{center}
$C_{H}\left( 53 \right) = J_{H}\left( 53 \right) \left( \Q \left(  \sqrt{53} \right) \right)_{\text{tors}} \cong \Z /7 \Z \times \Z / 91 \Z .$
\end{center}

\subsubsection*{p = 61}
A canonical model of $X_{H}\left( 61\right)$ is given by the zero locus of the following fifteen quadrics and one cubic.
\allowdisplaybreaks
\begin{align*}
 Q_{1} & = 4x_{1}x_{3} - 4x_{2}^2 - 404x_{4}x_{6} + 427x_{4}x_{7} - 329x_{4}x_{8} + 
        408x_{5}^2 - 459x_{5}x_{6} + 117x_{5}x_{7}  \\ & + 219x_{5}x_{8} + 240x_{6}^2 
        - 237x_{6}x_{7} - 539x_{6}x_{8} + 60x_{7}^2 + 469x_{7}x_{8} - 
        475x_{8}^2; \\
Q_{2} & = 4x_{1}x_{4} - 4x_{2}x_{3} + 16x_{4}x_{6} - 9x_{4}x_{7} + 23x_{4}x_{8} - 
        4x_{5}^2 + 9x_{5}x_{6} - 15x_{5}x_{7} - 5x_{5}x_{8} \\ & - 28x_{6}^2  + 
        39x_{6}x_{7} - 15x_{6}x_{8} - 8x_{7}^2 + 5x_{7}x_{8} + x_{8}^2;  \\
 Q_{3}  & = 4x_{1}x_{5} - 4x_{3}^2 - 316x_{4}x_{6} + 351x_{4}x_{7} - 249x_{4}x_{8} + 
        324x_{5}^2 - 347x_{5}x_{6} + 73x_{5}x_{7} \\ & + 199x_{5}x_{8}  + 168x_{6}^2  -
        169x_{6}x_{7} - 455x_{6}x_{8} + 44x_{7}^2 + 393x_{7}x_{8} - 379x_{8}^2;  \\
 Q_{4} & = 2x_{1}x_{6} - 2x_{3}x_{4} + 16x_{4}x_{6} - 11x_{4}x_{7} + 17x_{4}x_{8} - 
        8x_{5}^2 + 11x_{5}x_{6} - 5x_{5}x_{7} - 9x_{5}x_{8} \\ & - 14x_{6}^2   + 
        13x_{6}x_{7} + 13x_{6}x_{8} - 2x_{7}^2 - 9x_{7}x_{8} + 15x_{8}^2; \\
Q_{5} &=  4x_{1}x_{7} - 4x_{4}^2 - 216x_{4}x_{6} + 237x_{4}x_{7} - 175x_{4}x_{8} + 
        228x_{5}^2 - 241x_{5}x_{6} + 59x_{5}x_{7} \\ & + 129x_{5}x_{8}  + 128x_{6}^2 -
        139x_{6}x_{7} - 293x_{6}x_{8} + 36x_{7}^2 + 267x_{7}x_{8} - 253x_{8}^2; \\
  Q_{6} & = 2x_{1}x_{8} - 2x_{4}x_{5} + 34x_{4}x_{6} - 39x_{4}x_{7} + 25x_{4}x_{8} - 
        36x_{5}^2 + 39x_{5}x_{6} - 7x_{5}x_{7} - 23x_{5}x_{8} \\ & - 16x_{6}^2  + 
        15x_{6}x_{7} + 53x_{6}x_{8} - 4x_{7}^2 - 45x_{7}x_{8} + 45x_{8}^2; \\
 Q_{7} &= 4x_{2}x_{4} - 4x_{3}^2 - 232x_{4}x_{6} + 249x_{4}x_{7} - 183x_{4}x_{8} + 
        232x_{5}^2 - 237x_{5}x_{6} + 51x_{5}x_{7} \\ & + 149x_{5}x_{8} + 128x_{6}^2 -
        135x_{6}x_{7} - 309x_{6}x_{8} + 36x_{7}^2 + 271x_{7}x_{8} - 261x_{8}^2; \\
  Q_{8} & = 4x_{2}x_{5} - 4x_{3}x_{4} + 180x_{4}x_{6} - 189x_{4}x_{7} + 147x_{4}x_{8}- 
        180x_{5}^2 + 197x_{5}x_{6} - 51x_{5}x_{7} \\ & - 105x_{5}x_{8}  - 104x_{6}^2 +
        103x_{6}x_{7} + 245x_{6}x_{8} - 24x_{7}^2 - 211x_{7}x_{8} + 213x_{8}^2; \\
 Q_{9} &= 2x_{2}x_{6} - 2x_{4}^2 - 44x_{4}x_{6} + 49x_{4}x_{7} - 37x_{4}x_{8} + 
        44x_{5}^2 - 41x_{5}x_{6} + 11x_{5}x_{7} + 33x_{5}x_{8} \\ & + 26x_{6}^2  - 
        33x_{6}x_{7} - 55x_{6}x_{8} + 8x_{7}^2 + 53x_{7}x_{8} - 47x_{8}^2; \\
Q_{10} & = x_{2}x_{7} - x_{4}x_{5} + 16x_{4}x_{6} - 17x_{4}x_{7} + 13x_{4}x_{8} - 
        16x_{5}^2 + 20x_{5}x_{6} - 5x_{5}x_{7} - 7x_{5}x_{8} \\ & - 9x_{6}^2  + 
        8x_{6}x_{7}   + 22x_{6}x_{8} - 3x_{7}^2 - 18x_{7}x_{8} + 20x_{8}^2; \\
 Q_{11} & = 4x_{2}x_{8} - 16x_{4}x_{6} + 15x_{4}x_{7} - 13x_{4}x_{8} + 12x_{5}^2 - 
        19x_{5}x_{6} + 5x_{5}x_{7} + 3x_{5}x_{8} + 8x_{6}^2 \\ & - 5x_{6}x_{7}   - 
        27x_{6}x_{8} + 17x_{7}x_{8} - 23x_{8}^2; \\
 Q_{12} & = 2x_{3}x_{5} - 2x_{4}^2 - 90x_{4}x_{6} + 97x_{4}x_{7} - 75x_{4}x_{8} + 
        90x_{5}^2 - 97x_{5}x_{6} + 25x_{5}x_{7} + 53x_{5}x_{8} \\ & + 54x_{6}^2   - 
        57x_{6}x_{7} - 119x_{6}x_{8} + 14x_{7}^2 + 107x_{7}x_{8} - 105x_{8}^2 ;\\
   Q_{13} &= x_{3}x_{6} - x_{4}x_{5} - 22x_{4}x_{6} + 23x_{4}x_{7} - 18x_{4}x_{8} + 
        22x_{5}^2 - 23x_{5}x_{6} + 7x_{5}x_{7} + 14x_{5}x_{8} \\ & + 15x_{6}^2   - 
        17x_{6}x_{7} - 26x_{6}x_{8} + 4x_{7}^2 + 24x_{7}x_{8} - 23x_{8}^2; \\
 Q_{14} & = 4x_{3}x_{7} + 32x_{4}x_{6} - 39x_{4}x_{7} + 25x_{4}x_{8} - 36x_{5}^2 + 
        39x_{5}x_{6} - 5x_{5}x_{7} - 19x_{5}x_{8} - 8x_{6}^2 \\ & + x_{6}x_{7}  + 
        67x_{6}x_{8}  - 57x_{7}x_{8} + 51x_{8}^2; \\
 Q_{15} & = x_{3}x_{8} + 40x_{4}x_{6} - 43x_{4}x_{7} + 32x_{4}x_{8} - 40x_{5}^2 + 
        42x_{5}x_{6} - 10x_{5}x_{7} - 25x_{5}x_{8} - 23x_{6}^2 \\ & + 24x_{6}x_{7}  + 
        53x_{6}x_{8} - 6x_{7}^2 - 47x_{7}x_{8} + 45x_{8}^2; \\
 c & = 16256x_{1}^2x_{3} - 16256x_{1}x_{2}^2 + 16256x_{3}x_{4}^2 - 130048x_{4}^3 + 
        113792x_{4}^2x_{5} - 48768x_{4}x_{5}^2 \\ &  + 455620x_{4}x_{7}^2  + 
        934015x_{4}x_{7}x_{8} + 378671x_{4}x_{8}^2 + 81280x_{5}^3 + 
        130048x_{5}^2x_{6} \\ & + 567328x_{5}^2x_{7} + 1810692x_{5}^2x_{8}   - 
        1412640x_{5}x_{6}^2 - 710152x_{5}x_{6}x_{7} - 3915151x_{5}x_{6}x_{8} \\ &  + 
        499912x_{5}x_{7}^2 + 353577x_{5}x_{7}x_{8}  - 1485701x_{5}x_{8}^2 - 
        1474880x_{6}^3 + 3770532x_{6}^2x_{7} \\ & - 5789336x_{6}^2x_{8} - 
        2251600x_{6}x_{7}^2  + 5914379x_{6}x_{7}x_{8}  - 7757003x_{6}x_{8}^2 + 
        294568x_{7}^3 \\ &  - 467500x_{7}^2x_{8} + 3280949x_{7}x_{8}^2 - 
        3204251x_{8}^3.
 \end{align*}
The cusps of $X_{H}\left( 61 \right)$ are
\begin{align*}
    c_{1} & = \left(1 : 0 : 0 : 0 : 0 : 0 : 0 : 0 \right); \\
    c_{2}  & = \left(-3 : -2  :  0  : -1  :  -1  : -2 : -1 : 1 \right); \\
    c_{3} & = \left( 6 : \sqrt{61} + 13 : 6 : 2 : 2:  \sqrt{61} + 7 : \sqrt{61} + 7 : 2   \right); \\  
    c_{4}  & = \left( 6 : - \sqrt{61} + 13 : 6 : 2 : 2 :  -\sqrt{61} + 7 :  -\sqrt{61} + 7 : 2   \right).     \end{align*}
The cuspidal subgroup $C_{H}\left( 61 \right) = \langle \left[ c_{2} - c_{1} \right] ,\left[ c_{3} - c_{1} \right],\left[ c_{4} - c_{1} \right] \rangle$ is isomorphic to $ \left( \mathbb{Z} / 11 \mathbb{Z} \right) \times \left( \mathbb{Z} /55 \mathbb{Z} \right) $, and it fact, $ \left[ c_{3} - c_{1} \right] $ and $ \left[ c_{4} - c_{1} \right]$ are sufficient to generate the entire group. Taking Galois invariants, the resulting rational cuspidal subgroup is 
\begin{center}
$C_{H}\left( 61 \right)  \left( \mathbb{Q} \right) = \langle \left[ c_{3} + c_{4} - 2c_{1} \right] \rangle \cong  \left( \mathbb{Z} /55 \mathbb{Z} \right). $ 
\end{center}

Reducing modulo multiple primes, we find that $605$ is an upper bound for $\vert J_{H}\left( 61 \right) \left( \Q \left(  \sqrt{61} \right) \right)_{\text{tors}} $ and hence 
\begin{center}
$C_{H}\left( 61 \right) = J_{H}\left( 61 \right) \left( \Q \left(  \sqrt{61} \right) \right)_{\text{tors}} \cong \Z /11 \Z \times \Z / 55 \Z $.
\end{center}

\subsection{Curves of Genus $9$}
There is single prime congruent to $1$ modulo $4$ for which $X_{H}\left( p \right)$ has genus $9$, namely $73$. The curve $X_{H}\left(73\right)$ is not hyperelliptic, so its canonical model is the intersection of $21$ quadrics $\mathbb{P}^{8}$. 
\allowdisplaybreaks
\begin{align*}
Q_{1} & =  2x_{1}x_{3} - 2x_{2}^2 + 4x_{3}x_{4} - 3x_{4}^2 - 6x_{4}x_{5} + 29x_{5}^2 -  
        73x_{5}x_{6} - 219x_{5}x_{7}  + 219x_{5}x_{8} \\ & + 146x_{5}x_{9} + 73x_{6}^2  
        + 438x_{6}x_{7} - 438x_{6}x_{8} - 292x_{6}x_{9} + 584x_{7}^2 - 
        1314x_{7}x_{8} - 876x_{7}x_{9} \\ & + 657x_{8}^2   + 876x_{8}x_{9} + 
        292x_{9}^2; \\
 Q_{2} & = 2x_{1}x_{4} - 2x_{2}x_{3} - 2x_{3}x_{4} - 6x_{4}^2 + 7x_{4}x_{5} + 15x_{5}^2  - 73x_{5}x_{6} - 73x_{5}x_{7} + 219x_{5}x_{8} \\ & + 146x_{5}x_{9} + 
        73x_{6}^2 + 146x_{6}x_{7} - 438x_{6}x_{8} - 292x_{6}x_{9} + 73/2x_{7}^2 
        - 438x_{7}x_{8} - 292x_{7}x_{9} \\ &  + 657x_{8}^2 + 876x_{8}x_{9} + 
        292x_{9}^2; \\
  Q_{3} & =   4x_{1}x_{5} - 4x_{3}^2 - 8x_{3}x_{4} - 8x_{4}^2 + 16x_{4}x_{5} - 12x_{5}^2 + 
        146x_{5}x_{7} - 292x_{6}x_{7} - 511x_{7}^2 \\ & + 876x_{7}x_{8} + 
        584x_{7}x_{9}; \\ 
 Q_{4} & =  2x_{1}x_{6} - 2x_{3}x_{4} + 2x_{4}x_{5]} - 18x_{4}x_{8} - 6x_{5}^2 +  
        13x_{5}x_{6} + 59x_{5}x_{7} + 27x_{5}x_{8} + 32x_{5}x_{9} \\ & - 15x_{6}^2 - 
        177x_{6}x_{7} + 42x_{6}x_{8} - 34x_{6}x_{9} - 310x_{7}^2 + 411x_{7}x_{8} 
        + 172x_{7}x_{9} + 27x_{8}^2 \\ & + 222x_{8}x_{9} + 128x_{9}^2; \\
  Q_{5} & =  4x_{1}x_{7} - 4x_{4}^2 + 8x_{5}^2 - 20x_{5}x_{6} - 74x_{5}x_{7} + 
        24x_{5}x_{8} + 2x_{5}x_{9} + 18x_{6}^2 + 188x_{6}x_{7} \\ & - 72x_{6}x_{8} - 
        10x_{6}x_{9} + 280x_{7}^2 - 489x_{7}x_{8} - 244x_{7}x_{9} + 54x_{8}^2 - 
        42x_{8}x_{9} - 52x_{9}^2; \\
 Q_{6} & = 3x_{1}x_{8} - 3x_{4}x_{5} - 12x_{4}x_{8} + 6x_{5}^2 - 12x_{5}x_{6} - 
        28x_{5}x_{7} + 33x_{5}x_{8} + 9x_{5}x_{9} + 5x_{6}^2 \\ & + 60x_{6}x_{7} - 
        24x_{6}x_{8} + x_{6}x_{9} + 75x_{7}^2 - 171x_{7}x_{8} - 56x_{7}x_{9} + 
        36x_{8}^2 - 9x_{8}x_{9} - 26x_{9}^2; \\
 Q_{7} & = 4x_{1}x_{9} + 6x_{4}x_{8} - 4x_{5}^2 + 22x_{5}x_{6} - 10x_{5}x_{7} +  
        34x_{5}x_{9} - 14x_{6}^2 - 40x_{6}x_{7} + 24x_{6}x_{8} \\ & - 66x_{6}x_{9} - 
        18x_{7}^2 - 9x_{7}x_{8} - 194x_{7}x_{9} + 54x_{8}^2 + 330x_{8}x_{9} +  
        196x_{9}^2; \\ 
 Q_{8} & = 2x_{2}x_{4} - 2x_{3}^2 - 2x_{4}^2 - 10x_{4}x_{5} + 32x_{5}^2 - 73x_{5}x_{6} 
        - 219x_{5}x_{7} + 219x_{5}x_{8} + 146x_{5}x_{9} \\ & + 73x_{6}^2 + 
        438x_{6}x_{7} - 438x_{6}x_{8} - 292x_{6}x_{9} + 1241/2x_{7}^2 - 
        1314x_{7}x_{8} - 876x_{7}x_{9} \\ & + 657x_{8}^2 + 876x_{8}x_{9} + 
        292x_{9}^2; \\ 
  Q_{9} &= x_{2}x_{5} - x_{3}x_{4} - 2x_{4}^2 - x_{4}x_{5} + 11x_{5}^2 - 73/2x_{5}x_{6} - 
        73x_{5}x_{7} + 219/2x_{5}x_{8}+ 73x_{5}x_{9} \\ & + 73/2x_{6}^2 + 
        146x_{6}x_{7} - 219x_{6}x_{8} - 146x_{6}x_{9} + 146x_{7}^2 - 
        438x_{7}x_{8} - 292x_{7}x_{9} \\ & + 657/2x_{8}^2 + 438x_{8}x_{9} + 
        146x_{9}^2; \\
Q_{10} &= x_{2}x_{6} - x_{4}^2 - 9x_{4}x_{8} + 2x_{5}^2 - 6x_{5}x_{6} - x_{5}x_{7} + 
        15x_{5}x_{8} + 4x_{5}x_{9} + 2x_{6}^2 + 1/2x_{6}x_{7} \\ & - 6x_{6}x_{8} + 
        4x_{6}x_{9} - 8x_{7}^2 - 6x_{7}x_{8} + 30x_{7}x_{9} - 18x_{8}x_{9} - 
        16x_{9}^2; \\
 Q_{11} &=  4x_{2}x_{7} - 4x_{4}x_{5} + 8x_{5}^2 - 20x_{5}x_{6} - 50x_{5}x_{7} + 
        24x_{5}x_{8} + 2x_{5}x_{9} + 18x_{6}^2  + 132x_{6}x_{7}  \\ &  - 72x_{6}x_{8} - 
        10x_{6}x_{9} + 184x_{7}^2 - 345x_{7}x_{8} - 132x_{7}x_{9} + 54x_{8}^2 - 
        42x_{8}x_{9} - 52x_{9}^2; \\ 
 Q_{12} &= 3x_{2}x_{8} - 12x_{4}x_{8} - 3x_{5}^2 + 12x_{5}x_{6} + 34x_{5}x_{7} - 
        42x_{5}x_{8} - 39x_{5}x_{9} - 19x_{6}^2 - 64x_{6}x_{7} \\ & + 120x_{6}x_{8} +
        97x_{6}x_{9} - 73x_{7}^2 + 201x_{7}x_{8} + 192x_{7}x_{9} - 189x_{8}^2 - 
        297x_{8}x_{9} - 122x_{9}^2; \\
  Q_{13} &=  2x_{2}x_{9} + 3x_{4}x_{8} - 2x_{5}x_{6} + 3x_{5}x_{8} + 6x_{5}x_{9} + 
        4x_{6}^2 - 2x_{6}x_{7} - 18x_{6}x_{8} - 18x_{6}x_{9} \\ & - 6x_{7}^2 - 
        3x_{7}x_{8} - 11x_{7}x_{9} + 18x_{8}^2 + 42x_{8}x_{9} + 24x_{9}^2 ;\\
  Q_{14} &= 2x_{3}x_{5} - 2x_{4}^2 - 4x_{4}x_{5} + 8x_{5}^2 - 73x_{5}x_{7} + 
        146x_{6}x_{7} + 292x_{7}^2 - 438x_{7}x_{8} - 292x_{7}x_{9}; \\
  Q_{15} &=  2x_{3}x_{6} - 2x_{4}x_{5} + 12x_{4}x_{8} + 4x_{5}^2 - 10x_{5}x_{6} - 
        41x_{5}x_{7} - 6x_{5}x_{8} - 17x_{5}x_{9} + 15x_{6}^2 \\ & + 121x_{6}x_{7} - 
        54x_{6}x_{8} + 13x_{6}x_{9} + 185x_{7}^2 - 591/2x_{7}x_{8} - 
        118x_{7}x_{9} + 27x_{8}^2 \\ & - 111x_{8}x_{9} - 78x_{9}^2; \\
  Q_{16} &= 2x_{3}x_{7} - 2x_{5}^2 + 8x_{5}x_{6} + 7x_{5}x_{7} - 18x_{5}x_{8} - 
        9x_{5}x_{9} - 7x_{6}^2 - 14x_{6}x_{7}  + 36x_{6}x_{8} \\ & + 17x_{6}x_{9} -
        14x_{7}^2 + 69/2x_{7}x_{8} + 10x_{7}x_{9} - 45x_{8}^2 - 39x_{8}x_{9} - 
        6x_{9}^2; \\
   Q_{17} &=  6x_{3}x_{8} + 18x_{4}x_{8} - 6x_{5}x_{6} - 22x_{5}x_{7} - 30x_{5}x_{8} - 
        28x_{5}x_{9} + 16x_{6}^2 + 92x_{6}x_{7} \\ & - 36x_{6}x_{8} + 12x_{6}x_{9} + 
        140x_{7}^2 - 183x_{7}x_{8} - 68x_{7}x_{9} - 36x_{8}^2 - 156x_{8}x_{9} - 
        80x_{9}^2; \\
  Q_{18} &= 4x_{3}x_{9} - 6x_{4}x_{8} + 8x_{5}x_{7} + 6x_{5}x_{8} - 2x_{5}x_{9} -  
        6x_{6}^2 - 16x_{6}x_{7} + 18x_{6}x_{8} \\ & + 30x_{6}x_{9} - 50x_{7}^2 + 
        39x_{7}x_{8} + 66x_{7}x_{9} - 54x_{8}x_{9} - 36x_{9}^2 ;\\
  Q_{19} &= 2x_{4}x_{6} - 6x_{4}x_{8} - 2x_{5}^2 + 8x_{5}x_{6} + 15x_{5}x_{7} - 
        6x_{5}x_{8} - 3x_{5}x_{9} - 11x_{6}^2 - 50x_{6}x_{7} \\ &  + 42x_{6}x_{8} + 
        15x_{6}x_{9} - 53x_{7}^2 + 195/2x_{7}x_{8} + 42x_{7}x_{9} - 27x_{8}^2 + 
        3x_{8}x_{9} + 14x_{9}^2; \\
   Q_{20} &=  x_{4}x_{7} - x_{5}x_{6} + x_{5}x_{7} - 3x_{5}x_{8} - 4x_{5}x_{9} + x_{6}^2 + 
        2x_{6}x_{7} + 6x_{6}x_{9} + 2x_{7}^2 + 3x_{7}x_{8} \\ & + 16x_{7}x_{9} -
        9x_{8}^2 - 30x_{8}x_{9} - 16x_{9}^2; \\ 
 Q_{21} & = 2x_{4}x_{9} + x_{5}x_{7} - 3x_{5}x_{9} - x_{6}^2 - 4x_{6}x_{7} + 3x_{6}x_{8} + 5x_{6}x_{9} + x_{7}^2 - 3/2x_{7}x_{8} + 10x_{7}x_{9} \\ &  - 9x_{8}x_{9} - 
        6x_{9}^2.
\end{align*}
The cusps of $X_{H}\left( 73 \right)$ are
\begin{align*}
    c_{1} & = \left(1 : 0 : 0 : 0 : 0 : 0 : 0 : 0 : 0  \right); \\
    c_{2}  & = \left(3 : 0  :  0  : 0  :  0  : -3 : 0 :  -1 : 0 \right); \\
    c_{3} & = \left( 1 : 1 : 1 : 0: 0:  :\frac{1}{73} \left(  -11 \sqrt{73} -73  \right): -\frac{2}{73} \sqrt{73} : - \frac{4}{219} \sqrt{73} :   \frac{1}{146} \left( -11 \sqrt{73} - 73 \right)  \right); \\  
    c_{4}  & =  \left( 1 : 1 : 1 : 0: 0:  :\frac{1}{73} \left(  11 \sqrt{73} -73  \right): \frac{2}{73} \sqrt{73} : \frac{4}{219} \sqrt{73} :   \frac{1}{146} \left( 11 \sqrt{73} - 73 \right)  \right).
    \end{align*}

The cuspidal subgroup $C_{H}\left( 73 \right) = \langle \left[ c_{2} - c_{1} \right] ,\left[ c_{3} - c_{1} \right],\left[ c_{4} - c_{1} \right] \rangle$ is isomorphic to $ \left( \mathbb{Z} / 22 \mathbb{Z} \right) \times \left( \mathbb{Z} /66 \mathbb{Z} \right) $, and it fact, $ \left[ c_{3} - c_{1} \right] $ and $ \left[ c_{4} - c_{1} \right]$ are sufficient to generate the entire group. Taking Galois invariants, the resulting rational cuspidal subgroup is 
\begin{center}
$C_{H}\left( 73 \right)  \left( \mathbb{Q} \right) = \langle \left[ c_{3} + c_{4} - 2c_{1} \right] \rangle \cong  \left( \mathbb{Z} /66 \mathbb{Z} \right).$ 
\end{center}
The an upper bound for the size of $J_{H}\left( 73 \right) \left( \Q \left( \sqrt{73} \right) \right)_{\text{tors}}$ is $1452$, and hence 
\begin{center}
    $J_{H}\left( 73 \right) \left( \Q \left( \sqrt{73} \right) \right)_{\text{tors}} = C_{H}\left( 73 \right) \cong \Z /22 \Z \times \Z / 66 \Z. $
\end{center}

\bibliographystyle{abbrv}
\bibliography{ref}
\end{document}